\documentclass[article, 11pt]{amsart}

\usepackage{amsmath, amssymb, amsfonts, amsthm, latexsym}
\usepackage{tikz}

\newtheorem{Theorem}{Theorem}[section]
\newtheorem{Lemma}[Theorem]{Lemma}
\newtheorem{Corollary}[Theorem]{Corollary}
\newtheorem{Proposition}[Theorem]{Proposition}

\newtheorem{Question}[Theorem]{Question}

\theoremstyle{definition}

\newtheorem{Notation}[Theorem]{Notation}

\newcommand{\Z}{\mathbb{Z}}

\def\Ker{\operatorname{Ker}}
\def\ke{\operatorname{ker}}

\def\imm{\operatorname{im}}

\def\Ass{\operatorname{Ass}}
\def\Supp{\operatorname{Supp}}

\def\lim{\operatorname{lim}}

\def\ar{\operatorname{ar}}
\def\gr{\operatorname{gr}}

\def\mm{{\frak m}}
\def\pp{{\frak p}}

\begin{document}
	
	\title[Generalized Cohen-Macaulay rings]{Koszul Homology Under Small Perturbations}
	\author{Van Duc Trung}
	\address{Department of Mathematics, Uiversity of Education, Hue University, 34 Le Loi, Hue, Vietnam}
	\email{vdtrung@hueuni.edu.vn}
	
	
	\begin{abstract} Let $x_1,\ldots,x_s$ be a filter regular sequence in a local ring $(R,\mm)$. Denote by $R_{x_1,\ldots,x_s}$ the Koszul complex of $x_1,\ldots,x_s$ over $R$.  In this paper, we give an explicit number $N$ such that the sum of lengths $\sum_{i=1}^s (-1)^i\ell(H_i(R_{x_1,\ldots,x_s}))$ is preserved when we perturb the sequence $x_1, \ldots,x_s$ by $\varepsilon_1, \ldots, \varepsilon_s \in \mm^N$. Applying this result and the main Theorem of Eisenbud \cite{E}, we show that there exits $N >0$ such that for all $i \geq 1$ the length of $H_i(R_{x_1,\ldots,x_s})$  is preserved under small perturbation.
		
	\end{abstract}
	
	\maketitle
	
	\section{Introduction}
	This work is inspired by the recent intense activities in studying small perturbations. Taking a small perturbation arises naturally in studying deformations. One particular way to deform a singularity is by adding terms of high order to the defining equations. In this way we can transform a singularity defined analytically into an algebraic singularity by truncating the defining equations.
	
	Let $(R,\mm)$ be a Noetherian local ring, and let $I=(x_1,\ldots,x_s)$ be an ideal of $R$. We say that a sequence $x_1',\ldots,x_s'$ in 
	$R$ is a small perturbation of $x_1,\ldots,x_s$ (or of the ideal $I$) if, for each $i=1,\ldots,s$,
	$$x_i'-x_i \in  \mm^N$$
	for some fixed integer $N \gg 0$, where “sufficiently large” is understood to depend on the context. 
	
	The following natural question has been the subject of much investigation:
	
	\noindent \textbf{Question.} Which properties and invariants of $I$ are preserved under sufficiently small perturbations?
	
	In 1974, Eisenbud \cite{E} demonstrated how to control the homology of a complex under small perturbations. As an application, he showed that if $x_1,\ldots,x_s$ is a regular sequence, then the perturbed sequence $x_1+\varepsilon_1,\ldots,x_s+\varepsilon_s$ remains regular, provided the perturbation is sufficiently small. Huneke and Trivedi \cite{HT} extended this result to filter regular sequences. Recall that a sequence $x_1,\ldots,x_s$ is called filter regular if 
	$$\Supp \left( \frac{(x_1,\ldots,x_{i-1}) : x_i}{(x_1,\ldots,x_{i-1})} \right) \subseteq {\mm}$$
	for all $i = 1,\ldots,s$. In other words, the notion of a filter regular sequence generalizes the concept of a regular sequence.
	
	For numerical invariant, Srinivas and Trivedi \cite{ST1} showed that the Hilbert-Samuel function of $R/I$ with respect to an $\mm$-primary ideal $J$ does not change under small perturbations of $I$ if $R$ is a generalized Cohen-Macaulay ring and $I$ is generated by part of a system of parameters. Recently, Ma, Quy, and Smirnov \cite{MQS} extended this result to arbitrary local rings.
	
	The work of Ma, Quy and Smirnov \cite{MQS} leads to the question of whether the Hilbert-Samuel functions can be actually preserved under small perturbation. The least number $N$ such that the Hilbert-Samuel function of $R/I$ and $R/I'$ coincide for any $I' \equiv I$ mod $\mm^N$ is called the \textit{Hilbert perturbation index} of $R/I$. If $R$ is a Cohen-Macaulay ring, Srinivas and Trivedi \cite{ST2} gave an upper bound for the Hilbert perturbation index of $R/I$ in terms of the multiplicity. If $R$ is a generalized Cohen-Macaulay ring, inspired of the approach of Srinivas and Trivedi \cite{ST2}, Quy and the author \cite{QDT} gave an upper bound for the Hilbert perturbation index of $R/I$ in terms of the homological degree, a generalization of the multiplicity. However, this approach
	does not work in the general case, when R is an arbitrary local ring. Recently, by using a new approach, Quy and Trung \cite{QT} solved this problem in its generality.  
	
	Furthermore, Duarte \cite{LD1, LD2}  proved that the Betti numbers and local cohomology modules are also preserved by a sufficiently small perturbation under some suitable conditions.
	
	On the other hand, the Kozul complex is a fundamental notion and plays an important role in many areas of algebra and geometry. For this reason, it would be desirable to understand the behavior of Koszul complexes under small perturbations. 
	
	Let $x_1,\ldots,x_s$ be a sequence in the local ring $(R,\mm)$. We denote by $R_{x_1,\ldots,x_s}$ the Koszul complex of $x_1,\ldots,x_s$ over $R$. In this paper, we shall study the behavior of the Koszul homology modules $H_i(R_{x_1,\ldots,x_s})$ under small perturbations. It is well know that if $x_1,\ldots,x_s$ is a filter regular sequence, then $H_i(R_{x_1,\ldots,x_s})$ has finite length for $i=1,\ldots,s$. 
	
	Let $M$ be an arbitrary module over the local ring $(R,\mm)$. The Loewy length of $M$ is defined by
	$$\ell\ell_R(M) := \inf \{ n \geq 0  \ | \ \mm^nM=0 \}.$$
	
	If $I$ and $J$ are arbitrary ideals in the ring $R$, then by the  Artin-Rees lemma, there exists a number $c$ such that
	$$J^n \cap I = J^{n-c}(J^c\cap I)$$
	for all $n \geq c$. The smallest such number $c$ is called the Artin-Rees number of $I$ with respect to $J$, and is denoted by $\ar_J(I)$.
	
	Let $x_1,\ldots,x_s$ be a filter regular sequence in $R$. In \cite{QT}, the authors provide an explicit bound $N$, determined by the Loewy length and the Artin–Rees number, such that for every $\varepsilon_1, \ldots, \varepsilon_s \in \mm^N$, the perturbed sequence $x_1+\varepsilon_1,\ldots,x_s+\varepsilon_s$ remains filter regular. Building on the influential framework developed in \cite{QT}, which offers a powerful method for controlling perturbations of filter regular sequences, we establish the following main result.
	
	\noindent \textbf{Main Theorem 1} (Theorem \ref{main Theorem}). Let $x_1,\ldots,x_s$ be a filter regular sequence in $R$. For $i=1,\ldots,s$, let
	$$a_i = \ell\ell_R\left( \frac{(x_1,\ldots,x_{i-1}):x_i}{(x_1,\ldots,x_{i-1})}\right).$$
	Set $N=\max\{ a_1 + 2a_2 + \cdots + 2^{s-1}a_s, \ar_{\mm}(x_1),\ldots,\ar_{\mm}(x_1,\ldots,x_s)\} + 1$. Let $x_i' = x_i + \varepsilon_i$, where $\varepsilon_i$ is an arbitrary element in $\mm^N$, $i=1,\ldots,s$. Then
	$$\sum_{i=1}^s (-1)^i\ell(H_i(R_{x_1',\ldots,x_s'})) = \sum_{i=1}^s (-1)^i\ell(H_i(R_{x_1,\ldots,x_s})).$$

	From this result we obtain the following main result.
	
	\noindent \textbf{Main Theorem 2} (Theorem \ref{main Theorem 2}). Let $x_1,\ldots,x_s$ be a filter regular sequence in $R$. Then there exists a number $N$ such that for arbitrary elements $\varepsilon_1,\ldots,\varepsilon_s$ in $\mm^N$ we have
	$$\ell(H_i(R_{x_1',\ldots,x_s'})) = \ell(H_i(R_{x_1,\ldots,x_s})),$$
	for all $i \geq 1$, where $x_i'=x_i+\varepsilon_i, i=1,\ldots,s$.
	
	This paper is organized as follows. We prepare 
	some basic facts on Koszul complex and filter regular sequences in section 2. In section 3, beside the proof of the main Theorem we also give an upper bound for the Loewy length of Koszul homology modules.

	\noindent \textbf{Acknowledgements.} This work is supported by Vietnam National Program for the Development of Mathematics 2021-2030 under grant number B2023-CTT-03. It was initiated during a research stay of the author at Vietnam Institute for Advanced Study in Mathematics (VIASM) in 2023. The author sincerely thanks the referee for a thorough reading of the manuscript and for suggesting valuable improvements.

	\section{Preliminaries}
	
	Throughout this paper, we assume that all rings are commutative, noetherian with identity. In this section, we review some basic facts on Koszul complexes and filter regular sequences. 
	
	A complex $\underline{C}$ over ring $R$ is a sequence of $R$-modules an $R$-maps  
	$$\underline{C} = \cdots \rightarrow C_{n+1} \xrightarrow{f_{n+1}} C_n \xrightarrow{f_n} C_{n-1} \rightarrow \cdots, \quad n \in \Z, $$
	with $\imm f_{n+1} \subseteq \ke f_n$ for all $n$. The quotient module $\ke f_n/\imm f_{n+1}$ is called the $n$th homology module of $\underline{C}$, denoted by $H_n(\underline{C})$. Note that every $R$-module $M$ can be identified with the complex $\underline{C}$ such that $C_n = 0$ for all $n \neq 0$ and $C_0 = M$.
	
	Let $\underline{C}'$ and $\underline{C}''$ be complexes, the complex $\underline{C}' \otimes \underline{C}''$ is defined as follows: for all $n$
	$$(\underline{C}' \otimes \underline{C}'')_n = \sum_{p+q=n} C'_p \otimes C''_q,$$
	and the $n$th differentiation $f_n : (\underline{C}' \otimes \underline{C}'')_n \longrightarrow (\underline{C}' \otimes \underline{C}'')_{n-1}$ is given by
	
	$$f_n(c'_p \otimes c''_q) = f'_p(c'_p) \otimes c''_q + (-1)^pc'_p \otimes f''_q(c''_q).$$
	\vskip 0.2cm
	\noindent It is well known that $\underline{C}' \otimes \underline{C}'' \cong \underline{C}'' \otimes \underline{C}'$ and $\underline{C}' \otimes (\underline{C}'' \otimes \underline{C}''') \cong (\underline{C}' \otimes \underline{C}'')\otimes \underline{C}'''$.
	
	Let $R$ be a ring. The Koszul complex of an element $x$ in $R$ is defined by
	$$R_x : 0 \rightarrow C_1 \xrightarrow{x} C_0 \rightarrow 0,$$
	where $C_1 = C_0 = R$ and $C_1 \xrightarrow{x} C_0$ is multiplication by $x$. 
	
	Let $x_1,\ldots,x_s$ be a sequence in $R$. Since the complexes over $R$ form a multiplicative system under the operation of taking tensor products which is commutative and associative,  we define the Koszul complex of $x_1,\ldots,x_s$ over $R$ to be the complex
	$$R_{x_1,\ldots,x_s} = R_{x_1} \otimes \cdots \otimes R_{x_s}.$$
	
	Suppose $M$ is an $R$-module, the Koszul complex of $x_1,\ldots,x_s$ over $M$ is then defined by
	$$M_{x_1,\ldots,x_s} = M \otimes R_{x_1,\ldots,x_s}.$$
	
	\noindent We would like to mention here that the complex $M_{x_1,\ldots,x_s}$ is exactly the exterior algebra over $M$ generated by $x_1,\ldots,x_s$. The homology modules of the Koszul complex are called the Koszul homology modules. We note that $H_n(M_{x_1,\ldots,x_s}) = 0$ for $n <0$ and $n > s$. Moreover, it is clear to observe that 
	$$H_0(M_{x_1,\ldots,x_s})  = M/IM \ \text{and} \  H_s(M_{x_1,\ldots,x_s}) = (0:_M I),$$
	where $I$ is the ideal of $R$ generated by $x_1,\ldots,x_s$.
	
	We adopt the following property of Koszul complexes because it is useful for the next section (for more properties we refer to see \cite[section 6]{M} and  \cite[section 1]{MB}). Let $\underline{C}$ be a complex over $R$, we define the complex 
	$$\underline{C}_{x_1,\ldots,x_s} = \underline{C} \otimes R_{x_1,\ldots,x_s}.$$
	\begin{Proposition}\label{exact homology sequence 1} \cite[Prop 1.1]{MB}. Let $R$ be a ring. If $\underline{C}$ is a complex over $R$ and $x\in R$, then we have the following exact homology sequence 
		$$\cdots \rightarrow H_{n+1}(\underline{C}_x) \rightarrow H_n(\underline{C}) \xrightarrow{\partial_n} H_n(\underline{C}) \rightarrow H_n({\underline{C}}_x) \rightarrow \cdots$$ 
		where $\partial_n$ is multiplication by $(-1)^nx$.
	\end{Proposition} \label{exact homology sequence 2}
	Following Proposition \ref{exact homology sequence 1} we immediately obtain the following corollary.
	\begin{Corollary} \label{ex seq of homo} Let $R$ be a ring and $x_1,\ldots,x_s \in R$. For every $i = 1,\ldots,s$, we have the following exact sequence 
		$$\cdots \rightarrow H_{n+1}(R_{x_1,\ldots,x_s}) \rightarrow H_n(R_{x_1,\ldots,\widehat{x_i},\ldots,x_s}) \xrightarrow{\partial_n} H_n(R_{x_1,\ldots,\widehat{x_i},\ldots,x_s}) \rightarrow H_n(R_{x_1,\ldots,x_s}) \rightarrow \cdots$$
		where $\partial_n$ is multiplication by $(-1)^nx_i$.
	\end{Corollary}
	\begin{proof}
		Let $\underline{C} = R_{x_1,\ldots,\widehat{x_i},\ldots,x_s}$ and apply Proposition \ref{exact homology sequence 1} to $\underline{C}$ and $x_i$.
	\end{proof}
	We now recall the definition of filter regular sequences, along with some results related to their perturbations. From now on we assume $R$ is a local ring with maximal ideal $\mm$. The notion of filter regular sequence was introduced in \cite{CST}. An element $x \in R$ is called filter regular if $x \notin \pp$ for all $\pp \in \Ass(R) \setminus {\mm}$. It is easy to show that $x$ is filter regular if and only if $\ell\ell_R(0:x)$ is finite.
	
	A sequence of elements $x_1,\ldots,x_s$ is filter regular if for every $0 < i \leq s$ the image of $x_i$ in $R/(x_1,\ldots,x_{i-1})$ is a filter regular element. Note that if $x_1,\ldots,x_s$ is a filter regular sequence in $R$, then $H_n(R_{x_1,\ldots,x_s})$ has finite length for every $0 < n \leq s$ (see \cite[Lemma 1.17]{S}). It should also be cautioned that, unlike regular sequences, filter regular sequences do not generally permute. However, we have the following result:
	\begin{Lemma} \label{first filer regular element} \cite[Lemma 2.3]{MQS}.
		Let $x_1,\ldots,x_s$ be a filter regular sequence in $R$. Then the image of $x_1$ in $R/(x_2,\ldots,x_s)$ is a filter regular element.	
	\end{Lemma}
		
	Our approach is inspired by the method developed in \cite{QT}. We begin with the case of filter regular sequences of length one, for which we have the following result.
	\begin{Proposition} \label{filter regular element}\cite[Prop 3.2]{QT}.
		Let $x$ be a filter regular element in $R$. Set
		$$c = \max\{\ell\ell_R(0:x), \ar_{\mm}(x)+1\}.$$
		Let $x' = x + \varepsilon$, where $\varepsilon$ is an arbitrary element in $\mm^c$. Then $x'$ is a filter regular element with 
		$(0:x')=(0:x)$.
	\end{Proposition}
	Building on the above result and \cite[Theorem 3.5]{QT}, we establish the following result for filter regular sequences of length two.
	\begin{Proposition}\label{filter regular 2 elements}
		Let $x_1,x_2$ be a filter regular sequence in $R$. Set
		$$a_1 = \ell\ell_R(0:x_1), \ a_2 = \ell\ell_R\left( \frac{(x_1):x_2}{(x_1)}\right)$$
		and $c= \max\{ a_1+2a_2, \ar_{\mm}(x_1), \ar_{\mm}(x_1,x_2)\} +1$.  Let $x_i' = x_i + \varepsilon_i$, where $\varepsilon_i$ is an arbitrary element in $\mm^c$, $i=1,2$. Then
		\begin{enumerate}
			\item [\rm (i)] $x_1', x_2'$ is a filter regular sequence in $R$ with
			$$\ell\ell_R\left( \frac{(x_1'):x_2'}{(x_1')}\right) \leq 2a_2.$$
			\item [\rm (ii)] $$\ell \left( \frac{(x_1'):x_2'}{(x_1')} \right) = \ell\left( \frac{(x_1):x_2}{(x_1)} \right).$$
		\end{enumerate}
	\end{Proposition}
	\begin{proof}
		(i) See \cite[Theorem 3.5 (i)]{QT}.
		
		(ii) Consider the following short exact sequences
		$$ 0 \longrightarrow \frac{(x_1) + (0:x_2)}{(x_1)} \longrightarrow \frac{(x_1):x_2}{(x_1)} \longrightarrow \frac{(x_1):x_2}{(x_1) + (0:x_2)} \longrightarrow 0,$$
		$$ 0 \longrightarrow \frac{(x_1') + (0:x_2)}{(x_1')} \longrightarrow \frac{(x_1'):x_2}{(x_1')} \longrightarrow \frac{(x_1'):x_2}{(x_1') + (0:x_2)} \longrightarrow 0.$$
		Following the proof of \cite[Prop 3.4]{QT} we have
		$$\frac{(x_1') + (0:x_2)}{(x_1')} \cong \frac{(0:x_2)}{x_1(0:x_1x_2)} \cong \frac{(x_1) + (0:x_2)}{(x_1)} \quad \text{and} \quad \frac{(x_1'):x_2}{(x_1') + (0:x_2)} \cong \frac{(x_1):x_2}{(x_1) + (0:x_2)}.$$  
		Therefore, it follows from the above two short exact sequences that
		$$\ell \left( \frac{(x_1'):x_2}{(x_1')} \right) = \ell\left( \frac{(x_1):x_2}{(x_1)} \right).$$
		Set $\bar{R}=R/(x_1')$. It follows from the first assertion that $x_2$ is a filter regular element in $\bar{R}$ with $\ell\ell_{\bar{R}}(0 :_{\bar{R}} x_2) \leq 2a_2$. Furthermore, by \cite[Lemma 2.6]{QT} and \cite[Theorem 3.5 (iv)]{QT}, we have
		$$\ar_{\mm\bar{R}}(x_2\bar{R}) \leq \ar_{\mm}(x_1',x_2) = \ar_{\mm}(x_1,x_2).$$
		Now we can apply Proposition \ref{filter regular element} to the elements $x_2,x_2'$ in the ring $\bar{R}$ and get
		$$\frac{(x_1'):x_2}{(x_1')} = \frac{(x_1'):x_2'}{(x_1')}.$$
		Combining this fact with the above equality of lengths, we obtain the desired equality.
	\end{proof}

	\section{Koszul homology modules under small perturbations}
	In this section, we study the behavior of Koszul homology modules under small perturbations. We begin by examining the case of filter regular sequences of length one.	
	\begin{Proposition}
		Let $x$ be a filter regular element in $R$. Set
		$$c = \max\{a_{\mm}(0:x), \ar_{\mm}(x)+1\}.$$
		Let $x' = x + \varepsilon$, where $\varepsilon$ is an arbitrary element in $\mm^c$. Then 
		$$H_1(R_{x'}) = H_1(R_x).$$
	\end{Proposition}
	\begin{proof}
		Since $H_1(R_{x'}) = (0:x')$ and $H_1(R_{x}) = (0:x)$, the assertion follows from Proposition \ref{filter regular element}. 
	\end{proof}	
	For the last Koszul  homology module, we have the following result.
	\begin{Theorem} \label{last homology}
		Let $x_1,\ldots,x_s$ be a filter regular sequence in $R$. Set
		$$c = \max\{\ell\ell_R(0:x_1), \ar_{\mm}(x_1)+1\}.$$
		Let $x_i' = x_i + \varepsilon_i$, where $\varepsilon_i$ is an arbitrary element in $\mm^c$, $i=1,\ldots,s$. Then
		$$H_s(R_{x'_1,\ldots,x'_s}) = H_s(R_{x_1,\ldots,x_s}).$$
	\end{Theorem}
	\begin{proof}
		Let $I=(x_1,\ldots,x_s)$ and $I'=(x'_1,\ldots,x'_s)$.  It is well known that 
		$$H_s(R_{x_1,\ldots,x_s}) = (0 : I) \ \text{and} \ H_s(R_{x'_1,\ldots,x'_s}) = (0 : I').$$
		
		For every $x \in (0:I)$ we have $x\in(0:x_1)$. By Proposition \ref{filter regular element}, we have
		$$(0:x_1) = (0:x_1+\varepsilon_i),$$
		for every $i=1,\ldots,s$. It follows that $x\varepsilon_i =0$. Hence, $xx'_i = xx_i + x\varepsilon_i = 0$ for every $i=1,\ldots,s$ and $x\in(0:I')$. 
		
		Conversely, for every $x \in (0:I')$ we have $x\in(0:x'_1)=(0:x_1)$. It follows that $x\varepsilon_i =0$ for every $i=1,\ldots,s$. Hence, $xx_i = xx'_i - x\varepsilon_i = 0$ and $x\in(0:I)$. Thus, $(0:I') = (0:I)$ and the proof is complete.
	\end{proof}
	For the filter regular sequences of length two, we have the following result.
	\begin{Proposition}
		Let $x_1,x_2$ be a filter regular sequence in $R$. Let 
		$$a_1 = \ell\ell_R(0:x_1), \ a_2 = \ell\ell_R\left( \frac{(x_1):x_2}{(x_1)}\right).$$
		Set $N= \max\{ a_1+2a_2, \ar_{\mm}(x_1), \ar_{\mm}(x_1,x_2)\} +1$.  Let $x_i' = x_i + \varepsilon_i$, where $\varepsilon_i$ is an arbitrary element in $\mm^N$, $i=1,2$. Then
		\begin{enumerate}
			\item [\rm (i)] $H_2(R_{x_1,x_2}) = H_2(R_{x'_1,x'_2})$,
			\vskip 0.2cm
			\item [\rm (ii)] $\ell(H_1(R_{x_1,x_2})) = \ell(H_1(R_{x'_1,x'_2}))$.
		\end{enumerate}
	\end{Proposition}	
	\begin{proof}
		
		(i) The assertion is proved by Theorem \ref{last homology}.
		
		(ii) By Corollary \ref{ex seq of homo}, we have the following exact sequence
		$$0 \rightarrow H_2(R_{x_1,x_2}) \rightarrow H_1(R_{x_1}) \xrightarrow{-x_2} H_1(R_{x_1}) \rightarrow H_1(R_{x_1,x_2}) \rightarrow H_0(R_{x_1}) \xrightarrow{x_2} H_0(R_{x_1}).$$
		This yields an exact sequence
		$$0 \rightarrow H_2(R_{x_1,x_2}) \rightarrow H_1(R_{x_1}) \xrightarrow{-x_2} H_1(R_{x_1}) \rightarrow H_1(R_{x_1,x_2}) \rightarrow \Ker(H_0(R_{x_1}) \xrightarrow{x_2} H_0(R_{x_1})) \rightarrow 0.$$
		Since $$\Ker(H_0(R_{x_1}) \xrightarrow{x_2} H_0(R_{x_1})) = \frac{(x_1):x_2}{(x_1)},$$
		it follows from the above exact sequence that
		$$\ell(H_1(R_{x_1,x_2})) - \ell(H_2(R_{x_1,x_2}))=  \ell\left(\frac{(x_1):x_2}{(x_1)}\right).$$
		Similarly, we also have
		$$\ell(H_1(R_{x_1',x_2'})) - \ell(H_2(R_{x_1',x_2'})) =  \ell\left(\frac{(x_1'):x_2'}{(x_1')}\right).$$
		Hence, from the first assertion and Propopsition \ref{filter regular 2 elements} (ii), we obtain the desired equality.
	\end{proof}
	In order to extend to the filter regular sequences of higher length we need the following result.
	\begin{Proposition}\label{filter reular sequence}
		Let $x_1,\ldots,x_s$ be a filter regular sequence in $R$. For $i=1,\ldots,s$, let
		$$a_i = \ell\ell_R\left( \frac{(x_1,\ldots,x_{i-1}):x_i}{(x_1,\ldots,x_{i-1})}\right).$$
		Set $N=\max\{ a_1 + 2a_2 + \cdots + 2^{s-1}a_s, \ar_{\mm}(x_1),\ldots,\ar_{\mm}(x_1,\ldots,x_s)\} + 1$. Let $x_i' = x_i + \varepsilon_i$, where $\varepsilon_i$ is an arbitrary element in $\mm^N$, $i=1,\ldots,s$. Then
		\begin{enumerate}
			\item [\rm (i)] $x_1',\ldots, x_s'$ is a filter regular sequence in $R$ with
			$$\ell\ell_R\left( \frac{(x_1',\ldots,x_{s-1}'):x_s'}{(x_1',\ldots,x_{s-1}')}\right) \leq 2^{s-1}a_s.$$
			\item [\rm (ii)] $$\ell \left( \frac{(x_1',\ldots,x_{s-1}'):x_s'}{(x_1',\ldots,x_{s-1}')} \right) = \ell\left( \frac{(x_1,\ldots,x_{s-1}):x_s}{(x_1,\ldots,x_{s-1})} \right).$$
		\end{enumerate}
	\end{Proposition}
	\begin{proof}
		(i) is proved in \cite[Theorem 3.5 (i)]{QT}.
		
		(ii) We will proceed by induction on $s$. The cases $s=1,2$ are proved by Proposition \ref{filter regular element} and Proposition \ref{filter regular 2 elements}. Suppose now that $s >2$.
		
		By Lemma \ref{first filer regular element}, we have $x_1,x_s$ is a filter regular sequence in $R_1 = R/(x_2,\ldots,x_{s-1})$. Following the proof of \cite[Theorem 3.5]{QT} we have
		$$\ell\ell_{R_1}(0:_{R_1} x_1) \leq \ell\ell_R\left(\frac{(x_2,\ldots,x_{s-1}):x_1}{(x_2,\ldots,x_{s-1})}\right) \leq a_1 + \cdots + a_{s-1}.$$
		By \cite[Lemma 2.6]{QT}, we also have the following
		\begin{align*}
		& \ar_{\mm}(x_1R_1) \leq \ar_{\mm}(x_1,\ldots,x_{s-1}),\\ 
		& \ar_{\mm}(x_1,x_s)R_1 \leq \ar_{\mm}(x_1,\ldots,x_s).
		\end{align*}
		Now we can apply Proposition \ref{filter regular 2 elements} to the filter regular sequence  $x_1,x_s$ in the ring $R_1$ and get
		$$\ell \left( \frac{x_1'R_1:x_s}{x_1'R_1} \right) = \ell\left( \frac{x_1R_1:x_s}{x_1R_1} \right).$$
		This is equivalent to 
		\begin{equation}
			\ell \left( \frac{(x_1',x_2,\ldots,x_{s-1}):x_s}{(x_1',x_2,\ldots,x_{s-1})} \right) = \ell\left( \frac{(x_1,x_2,\ldots,x_{s-1}):x_s}{(x_1,x_2,\ldots,x_{s-1})} \right).
		\end{equation}	
		
		On the other hand, by \cite[Theorem 3.5 (i)]{QT}, $x_1',x_2,\ldots,x_s$ is a filter regular sequence in $R$. It follows that $x_2,\ldots,x_s$ is  a filter regular sequence in $R'=R/(x_1')$.
		
		Following the proof of \cite[Theorem 3.5]{QT} we have
		$$\bar{a}_i= \ell\ell_{R'}\left(\frac{(x_2,\ldots,x_{i-1})R':x_i}{(x_2,\ldots,x_{i-1})R'}\right) \leq \ell\ell_R\left(\frac{(x_1',x_2,\ldots,x_{i-1}):x_i}{(x_1',x_2,\ldots,x_{i-1})}\right)\leq 2a_i,$$
		for $i=2,\ldots,s$. It follows that
		$$\bar{a}_2+2\bar{a}_3+\cdots+2^{s-2}\bar{a}_s \leq 2a_2+2^2a_3+\cdots+2^{s-1}a_s.$$
		By \cite[Lemma 2.6]{QT} and \cite[Theorem 3.5(iv)]{QT}, we also have
		$$\ar_{\mm}(x_2,\ldots,x_i)R' \leq \ar_{\mm}(x_1',x_2,\ldots,x_i) = \ar_{\mm}(x_1,x_2,\ldots,x_i),$$
		for $i=2,\ldots,s$. By induction, one has
		$$\ell \left( \frac{(x_2',\ldots,x_{s-1}')R':x_s'}{(x_2',\ldots,x_{s-1}')R'} \right) = \ell\left( \frac{(x_2,\ldots,x_{s-1})R':x_s}{(x_2,\ldots,x_{s-1})R'} \right).$$
		This is equivalent to
		\begin{equation}
			\ell \left( \frac{(x_1',x_2',\ldots,x_{s-1}'):x_s'}{(x_1',x_2',\ldots,x_{s-1}')} \right) = \ell\left( \frac{(x_1',x_2,\ldots,x_{s-1}):x_s}{(x_1',x_2,\ldots,x_{s-1})} \right).
		\end{equation}
		Combining the equations (1) and (2), we obtain the desired equality.
	\end{proof}
	We now proceed to state the main results of this paper.
	\begin{Theorem}\label{main Theorem}
		Let $x_1,\ldots,x_s$ be a filter regular sequence in $R$. For $i=1,\ldots,s$, let
		$$a_i = \ell\ell_R\left( \frac{(x_1,\ldots,x_{i-1}):x_i}{(x_1,\ldots,x_{i-1})}\right).$$
		Set $N=\max\{ a_1 + 2a_2 + \cdots + 2^{s-1}a_s, \ar_{\mm}(x_1),\ldots,\ar_{\mm}(x_1,\ldots,x_s)\} + 1$. Let $x_i' = x_i + \varepsilon_i$, where $\varepsilon_i$ is an arbitrary element in $\mm^N$, $i=1,\ldots,s$. Then
		$$\sum_{i=1}^s (-1)^i\ell(H_i(R_{x_1',\ldots,x_s'})) = \sum_{i=1}^s (-1)^i\ell(H_i(R_{x_1,\ldots,x_s})).$$
	\end{Theorem}
	\begin{proof}
		By Corollary \ref{ex seq of homo}, we have the following exact sequence
		$$0 \rightarrow H_s(R_{x_1,\ldots,x_s}) \rightarrow H_{s-1}(R_{x_1,\ldots,x_{s-1}}) \xrightarrow{(-1)^{s-1}x_s} H_{s-1}(R_{x_1,\ldots,x_{s-1}}) \rightarrow H_{s-1}(R_{x_1,\ldots,x_s}) \rightarrow \cdots$$ 
		$$\cdots \rightarrow H_1(R_{x_1,\ldots,x_s}) \rightarrow \Ker(H_0(R_{x_1,\ldots,x_{s-1}}) \xrightarrow{x_s} H_0(R_{x_1,\ldots,x_{s-1}})) \rightarrow 0.$$
		Since
		$$\Ker(H_0(R_{x_1,\ldots,x_{s-1}}) \xrightarrow{x_s} H_0(R_{x_1,\ldots,x_{s-1}})) = \frac{(x_1,\ldots,x_{s-1}):x_s}{(x_1,\ldots,x_{s-1})},$$
		it follows from the exact sequence that
		$$\sum_{i=1}^s (-1)^i\ell(H_i(R_{x_1,\ldots,x_s})) = -\ell\left(\frac{(x_1,\ldots,x_{s-1}):x_s}{(x_1,\ldots,x_{s-1})} \right).$$
		Similarly, we also have
		$$\sum_{i=1}^s (-1)^i\ell(H_i(R_{x_1',\ldots,x_s'})) = -\ell\left(\frac{(x_1',\ldots,x_{s-1}'):x_s'}{(x_1',\ldots,x_{s-1}')} \right).$$
		Hence, from Proposition \ref{filter reular sequence} (ii), we obtain the desired equality.
	\end{proof}
	As a natural next step, we are led to consider the following question.
	
	\begin{Question} \label{Q1} Let $x_1,\ldots,x_s$ be a filter regular sequence in $R$. Does there exist a number $N$ such that for arbitrary elements $\varepsilon_1,\ldots,\varepsilon_s$ in $\mm^N$ we have
		$$\ell(H_i(R_{x_1',\ldots,x_s'})) = \ell(H_i(R_{x_1,\ldots,x_s})),$$
		for all $i  \geq 1$, where $x_i'=x_i+\varepsilon_i, i=1,\ldots,s$? 
	\end{Question}
	
	We now recall the main result from \cite{E}, which is particularly useful for addressing Question~\ref{Q1}. Let
	$$\underline{C} = \cdots \rightarrow C_1 \xrightarrow{f_1} C_0 \xrightarrow{f_0} C_{-1} \xrightarrow{f_{-1}} \cdots$$
	be a complex of finitely generated $R$-modules. An $\mm$-adic approximation of $\underline{C}$ of order $d = (\ldots,d_1,d_0,d_{-1},\ldots)$ is a complex
	$$\underline{C}_{\varepsilon} = \cdots \rightarrow C_1 \xrightarrow{f_1+\varepsilon_1} C_0 \xrightarrow{f_0+\varepsilon_0} C_{-1} \xrightarrow{f_{-1}+\varepsilon_{-1}} \cdots$$
	where, for each $i$, the map
	$$\varepsilon_i : C_i \rightarrow \mm^{d_i}C_{i-1}.$$
	Suppose the complex
	$$\underline{C} = \cdots \rightarrow C_{n+1} \xrightarrow{f_{n+1}} C_n \xrightarrow{f_n} C_{n-1} \xrightarrow{f_{n-1}} \cdots$$
	is filtered by subcomplexes
	$$\underline{C} = \underline{C}_0 \supseteq \underline{C}_1 \supseteq \cdots$$
	where $\underline{C}_p = \mm^p\underline{C}$. This filtration induces a corresponding filtration on the homology $H(\underline{C})$. We denote by $H_n(\underline{C})_p$ the $p$th submodule in the induced filtration. By \cite[Section 2A.5]{S1}, we have $\bigcap_{p \geq 0} H_n(\underline{C})_p = 0$, and for sufficiently large $p$, $H_n(\underline{C})_{p+1} = \mm H_n(\underline{C})_p$. This allows us to define the associated graded module $\gr H_n(\underline{C})$.
	
	\begin{Theorem}\label{maintheoremofEisenbud}\cite[Main Theorem]{E} With the above notation, there exists a sequence of positive integers $d = (\ldots,d_1,d_0,d_{-1},\ldots)$ such that if $\underline{C}_{\varepsilon}$ is an $\mm$-adic approximation of $\underline{C}$ of order $d$, then
	\begin{enumerate}
		\item [\rm (i)] $\gr H_n(\underline{C}_{\varepsilon})$ is a subquotient of $\gr H_n(\underline{C})$ for all $n$, and
		\vskip 0.2cm
		\item [\rm (ii)] if $H_n(\underline{C})$ and $H_{n-1}(\underline{C})$ are both annihilated by some power of $\mm$, then $$\gr H_n(\underline{C}) \cong \gr H_n(\underline{C}_{\varepsilon}).$$
	\end{enumerate}	
	
	\end{Theorem}
	\noindent In fact, \cite{E} treats the case of an arbitrary ring $R$, and the approximation is formulated $J$-adically, where $J$ is an ideal contained in the Jacobson radical of $R$.
	
	Let $R_{x_1,\ldots,x_s}$ be the Koszul complex of a filter regular sequence $x_1,\ldots,x_s$. Given $N > 0$, let $x_i' = x_i + \varepsilon_i$, where $\varepsilon_i$ is an arbitrary element in $\mm^N$, $i=1,\ldots,s$. In the sense of \cite{E}, the complex $R_{x_1',\ldots,x_s'}$ is an approximation of order $(\ldots,N,N,N,\ldots)$ of $R_{x_1,\ldots,x_s}$. From Theorem \ref{maintheoremofEisenbud}, we derive the following:
	\begin{Lemma}\label{H2} Let $x_1,\ldots,x_s$ be a filter regular sequence in $R$. Then there exists a number $N$ such that for arbitrary elements $\varepsilon_1,\ldots,\varepsilon_s$ in $\mm^N$ we have
		$$\ell(H_i(R_{x_1',\ldots,x_s'})) = \ell(H_i(R_{x_1,\ldots,x_s})),$$
		for all  $i \geq 2$, where $x_i'=x_i+\varepsilon_i, i=1,\ldots,s.$
	\end{Lemma}
	\begin{proof}
		Since $H_i(R_{x_1,\ldots,x_s})$ and $H_{i-1}(R_{x_1,\ldots,x_s})$ have finite length for all $i \geq 2$, it follows from Theorem \ref{maintheoremofEisenbud} that 
		$$\ell(H_i(R_{x_1',\ldots,x_s'})) = \ell(H_i(R_{x_1,\ldots,x_s})),$$
		for all  $i \geq 2$.
	\end{proof}
	
	Combining Theorem \ref{main Theorem} with Lemma \ref{H2}, we obtain an  affirmative answer to Question \ref{Q1}.
	
	\begin{Theorem}\label{main Theorem 2} Let $x_1,\ldots,x_s$ be a filter regular sequence in $R$. Then there exists a number $N$ such that for arbitrary elements $\varepsilon_1,\ldots,\varepsilon_s$ in $\mm^N$ we have
		$$\ell(H_i(R_{x_1',\ldots,x_s'})) = \ell(H_i(R_{x_1,\ldots,x_s})),$$
		for all $i\geq 1$, where $x_i'=x_i+\varepsilon_i, i=1,\ldots,s.$
	\end{Theorem}

	Let $(R,\mm)$ be a local ring. If $c$ is the Loewy length of $R$-module $M$ then $\varepsilon M = 0$ for every $\varepsilon \in \mm^c$. Therefore, in the study of small perturbations, it would be nice if we can give an upper bound for the Loewy length. We close this section by giving an upper bound for Loewy length of Koszul homology modules.
	
	For the first Koszul homology modules we have the following result.
	\begin{Proposition}\label{first loewy length}
		Let $x_1,\ldots,x_s$ be a filter regular sequence in $R$. For $i=1,\ldots,s$, let
		$$a_i = \ell\ell_R\left( \frac{(x_1,\ldots,x_{i-1}):x_i}{(x_1,\ldots,x_{i-1})}\right).$$
		Set $N=\max\{ a_1 + 2a_2 + \cdots + 2^{s-1}a_s, \ar_{\mm}(x_1),\ldots,\ar_{\mm}(x_1,\ldots,x_s)\} + 1$. Let $x_i' = x_i + \varepsilon_i$, where $\varepsilon_i$ is an arbitrary element in $\mm^N$, $i=1,\ldots,s$. Then
		$$\ell\ell_R(H_1(R_{x_1',\ldots,x_s'})) \leq a_1 + 2a_2 + \cdots + 2^{s-1}a_s.$$ 
	\end{Proposition}
	\begin{proof}
		We will proceed by induction on $s$. For $s=1$, one has $\ell\ell_R(H_1(R_{x_1'})) = \ell\ell_R(0:x_1') = \ell\ell_R(0:x_1) = a_1.$
		
		Suppose $s >1$. By Corollary \ref{ex seq of homo}, we have the following exact sequence
		$$H_1(R_{x_1',\ldots,x_{s-1}'}) \xrightarrow{-x_s'} H_1(R_{x_1',\ldots,x_{s-1}'}) \rightarrow H_1(R_{x_1',\ldots,x_s'}) \rightarrow H_0(R_{x_1',\ldots,x_{s-1}'}) \xrightarrow{x_s'} H_0(R_{x_1',\ldots,x_{s-1}'}).$$
		This yields a short exact sequence
		$$0 \rightarrow H_1(R_{x_1',\ldots,x_{s-1}'})/x_s'H_1(R_{x_1',\ldots,x_{s-1}'}) \rightarrow H_1(R_{x_1',\ldots,x_s'}) \rightarrow \frac{(x_1',\ldots,x_{s-1}'):x_s'}{(x_1',\ldots,x_{s-1}')} \rightarrow 0.$$
		From this exact sequence and by induction we obtain
		\begin{align*}
			\ell\ell_R(H_1(R_{x_1',\ldots,x_s'})) & \leq \ell\ell_R(H_1(R_{x_1',\ldots,x_{s-1}'})) + \ell\ell_R\left( \frac{(x_1',\ldots,x_{s-1}'):x_s'}{(x_1',\ldots,x_{s-1}')}\right) \\
			&  \leq a_1 + 2a_2 + \cdots + 2^{s-2}a_{s-1} + 2^{s-1}a_s.
		\end{align*}
		The proof is complete.
	\end{proof}
	In order to construct the bound for Loewy length of the other Koszul homology modules we need the following notation.
	
	\begin{Notation} With the notation as in Proposition \ref{first loewy length} we define
		\begin{align*}
			n_1(i) &= a_1 + 2a_2 + \cdots + 2^{i-1}a_i, \ i=1,\ldots,s \\
			n_2(i) &= n_1(1) + n_1(2) + \cdots + n_1(i), \ i=1,\ldots,s \\
			\cdots & \\
			n_s(i) &= n_{s-1}(1) + n_{s-1}(2) + \cdots + n_{s-1}(i),\  i=1,\ldots,s. 
		\end{align*}
	\end{Notation}
	We note that these notations only depend on $a_1,\ldots,a_s$.
	\begin{Theorem}
		Let $x_1,\ldots,x_s$ be a filter regular sequence in $R$. For $i=1,\ldots,s$, let
		$$a_i = \ell\ell_R\left( \frac{(x_1,\ldots,x_{i-1}):x_i}{(x_1,\ldots,x_{i-1})}\right).$$
		Set $N=\max\{ a_1 + 2a_2 + \cdots + 2^{s-1}a_s, \ar_{\mm}(x_1),\ldots,\ar_{\mm}(x_1,\ldots,x_s)\} + 1$. Let $x_i' = x_i + \varepsilon_i$, where $\varepsilon_i$ is an arbitrary element in $\mm^N$, $i=1,\ldots,s$. Then
		$$\ell\ell_R(H_k(R_{x_1',\ldots,x_s'})) \leq n_k(s-k+1),$$
		for $k = 1,\ldots,s$.
	\end{Theorem}
	\begin{proof}
		We will proceed by induction on $s$. The case $s=1$ is obvious.
		
		Suppose $s>1$. By Proposition \ref{first loewy length}, the assertion holds for $k=1$. For $k>1$, by Corollary \ref{ex seq of homo}, we have the following exact sequence.
		\begin{align*}
			H_k(R_{x_1',\ldots,x_{s-1}'}) \xrightarrow{(-1)^kx_s'} H_k(R_{x_1',\ldots,x_{s-1}'}) &\rightarrow H_k(R_{x_1',\ldots,x_s'}) \\
			&\rightarrow H_{k-1}(R_{x_1',\ldots,x_{s-1}'}) \xrightarrow{(-1)^{k-1}x_s'} H_{k-1}(R_{x_1',\ldots,x_{s-1}'}).
		\end{align*}
		This yields a short exact sequence
		$$0 \rightarrow H_k(R_{x_1',\ldots,x_{s-1}'})/x_s'H_k(R_{x_1',\ldots,x_{s-1}'}) \rightarrow H_k(R_{x_1',\ldots,x_s'}) \rightarrow (0 :_{H_{k-1}(R_{x_1',\ldots,x_{s-1}'})} x_s') \rightarrow 0.$$
		From this exact sequence and by induction we obtain
		\begin{align*}
			\ell\ell_R(H_k(R_{x_1',\ldots,x_s'})) & \leq \ell\ell_R(H_k(R_{x_1',\ldots,x_{s-1}'})) + \ell\ell_R(H_{k-1}(R_{x_1',\ldots,x_{s-1}'})) \\
			& \leq n_k(s-k) + n_{k-1}(s-k+1) = n_k(s-k+1).
		\end{align*}	
		The proof is complete.
	\end{proof}

\end{document}